\documentclass[a4paper,12pt]{article}
\usepackage[top=2.5cm,bottom=2.5cm,left=2.5cm,right=2.5cm]{geometry}
\usepackage{cite,color,amsmath,amssymb,amsthm,graphicx}
\usepackage[margin=1cm,%
                font=small,%
                format=hang,%
                labelsep=period,%
                labelfont=bf]{caption}
\usepackage{longtable,pdflscape,booktabs,caption,multicol}
\usepackage[colorlinks=true,citecolor=black,linkcolor=black,urlcolor=blue]{hyperref}
\usepackage{float}
\usepackage{tikz}
\usetikzlibrary{arrows}
\usetikzlibrary{arrows.meta}
\usetikzlibrary{chains}
\usetikzlibrary{positioning}
\usetikzlibrary{automata,positioning,calc}
\usetikzlibrary{decorations}
\usetikzlibrary{decorations.shapes}
\usetikzlibrary{decorations.markings}
\tikzset{
    edge/.style={-{Latex[scale=1.7]}},
    dedge/.style={{Latex[scale=1.7]}-{Latex[scale=1.7]}},
}

\theoremstyle{plain}
\newtheorem{theorem}{Theorem}
\newtheorem{lemma}[theorem]{Lemma}

\theoremstyle{definition}

\theoremstyle{remark}

\usepackage{pifont}
\makeatletter
\def\@fnsymbol#1{\ensuremath{\ifcase#1\or $\ding{73}$\or *\or \dagger\or \ddagger\or
   \mathsection\or \mathparagraph\or \|\or **\or \dagger\dagger
   \or \ddagger\ddagger \else\@ctrerr\fi}}
\makeatother

\begin{document}
\phantom{.}
\vskip 6cm

\begin{center}
{\Large \bf Greedy trees have minimum Sombor indices\footnote{%
The first author is supported by Diffine LLC.
The second author is supported by the Serbian Ministry of Education, Science and Technological Development 
through the Mathematical Institute of SASA, and by the project F-159 of the Serbian Academy of Sciences and Arts.}
}

\bigskip\bigskip
{\bf Ivan Damnjanovi\'c$^{a,b}$\footnote{Corresponding author.}, Dragan Stevanovi\'c$^c$}

\bigskip
{\small \em
$^a$University of Ni\v s, Faculty of Electronic Engineering, Ni\v s, Serbia

$^b$Diffine LLC

$^c$Mathematical Institute, Serbian Academy of Sciences and Arts, Belgrade, Serbia
}

\bigskip
{\tt ivan.damnjanovic@elfak.ni.ac.rs, ivan@diffine.com, dragan\_stevanovic@mi.sanu.ac.rs}

\bigskip
(Received November \dots, 2022)
\end{center}

\begin{abstract}
Recently, Gutman [MATCH Commun.\ Math.\ Comput.\ Chem.\ \textbf{86} (2021) 11--16] defined a new graph invariant 
which is named the Sombor index $\mathrm{SO}(G)$ of a graph~$G$ and is computed via the expression
\[
    \mathrm{SO}(G) = \sum_{u \sim v} \sqrt{\deg(u)^2 + \deg(v)^2} ,
\]
where $\deg(u)$ represents the degree of the vertex $u$ in $G$ and 
the summing is performed across all the unordered pairs of adjacent vertices $u$ and $v$. 
Here we take into consideration the set of all the trees $\mathcal{T}_D$ that have a specified degree sequence~$D$ and 
show that the greedy tree attains the minimum Sombor index on the set $\mathcal{T}_D$.
\end{abstract}


\baselineskip=0.30in
\section{Introduction}

In this paper we will consider all graphs to be undirected, finite, simple and non-null. Thus, every graph will have at least one vertex and there shall be no loops or multiple edges. For convenience we will take that each graph of order $n$ has the vertex set $\{1, 2, 3, \ldots, n\}$.

Furthermore, for a given graph $G$ of order $n$ and any $u = \overline{1, n}$, we shall use the notation $\deg(u)$ to signify the degree of the vertex $u$, i.e.\ the total number of vertices adjacent to it. Taking this into consideration, it is possible to define the Sombor index $\mathrm{SO}(G)$ of the graph $G$ by using the expression
\[
    \mathrm{SO}(G) = \sum_{u \sim v} \sqrt{\deg(u)^2 + \deg(v)^2} ,
\]
where the summing is done across all the unordered pairs of adjacent vertices $u$ and $v$, as done so by Gutman \cite{Gutman}.
Although it was defined very recently, the Sombor index already managed to attract a lot of attention from researchers---%
see \cite{2,3,4,5,6,7,8,9,10,11,12,13,14,15,16,17,18,19,20,21} for a partial list of results on the Sombor index. 

For a given $n \in \mathbb{N}$, 
let $D = (d_1, d_2, \ldots, d_n)$ be an arbitrary non-increasing sequence of $n$ non-negative integers.
We shall use $\mathcal{T}_D$ to denote the set of all the trees of order $n$ such that $D$ represents their degree sequence.
For convenience, we shall take into consideration only the trees such that $d_u=\deg(u)$ for each $u = \overline{1, n}$. 
The reason why this can be done is clear---%
all the other trees that adhere to the degree sequence~$D$ are surely isomorphic to at least one aforementioned tree.

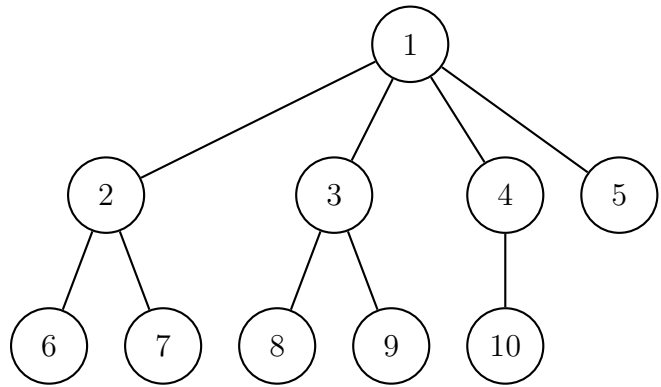
\begin{figure}[h]
    \centering
    \begin{tikzpicture}
        \node[state, minimum size=1cm, thick] (1) at (0, 0) {$6$};
        \node[state, minimum size=1cm, thick] (2) at (1.5, 0) {$7$};
        \node[state, minimum size=1cm, thick] (3) at (3.0, 0) {$8$};
        \node[state, minimum size=1cm, thick] (4) at (4.5, 0) {$9$};
        \node[state, minimum size=1cm, thick] (5) at (6.0, 0) {$10$};
        
        \node[state, minimum size=1cm, thick] (6) at (0.75, 2) {$2$};
        \node[state, minimum size=1cm, thick] (7) at (3.75, 2) {$3$};
        \node[state, minimum size=1cm, thick] (8) at (6.0, 2) {$4$};
        \node[state, minimum size=1cm, thick] (9) at (7.5, 2) {$5$};

        \node[state, minimum size=1cm, thick] (10) at (4.75, 4) {$1$};

        \path[thick] (1) edge (6);
        \path[thick] (2) edge (6);
        \path[thick] (3) edge (7);
        \path[thick] (4) edge (7);
        \path[thick] (5) edge (8);
        \path[thick] (6) edge (10);
        \path[thick] (7) edge (10);
        \path[thick] (8) edge (10);
        \path[thick] (9) edge (10);
    \end{tikzpicture}
    \caption{The greedy tree $GT_D$ for $D = (4, 3, 3, 2, 1, 1, 1, 1, 1, 1)$.}
    \label{greedy_example}
\end{figure}

We define the greedy tree $GT_D$ as the unique rooted tree from $\mathcal{T}_D$ 
such that its breadth-first traversal yields the sequence $(1, 2, 3, \ldots, n)$.
In other words,
the root~1 has $d_1$ children $2,3,\dots,d_1+1$,
its child~2 has $d_2-1$ children $d_1+2,\dots,d_1+d_2$, and so on.
An example of a greedy tree is given in Figure~\ref{greedy_example}
for the degree sequence $D = (4, 3, 3, 2, 1, 1, 1, 1, 1, 1)$.
The greedy tree~$GT_D$ often appears as an extremal tree in~$\mathcal{T}_D$:
for example, it minimizes 
Wiener index~\cite{22}, the incidence energy~\cite{22} and the sum of vertex eccentricities~\cite{23},
while it maximizes 
the connectivity and sum-connectivity indices~\cite{24},
the spectral moments~\cite{25},
the spectral radius of generalized reverse distance matrix~\cite{26},
the number of pairs of vertices whose distance is at most~$k$ for arbitrary~$k$ \cite{27}
and the number of subtrees of any given order~\cite{22,28}.
Our goal here is to prove the following theorem
which shows that $GT_D$ also minimizes the Sombor index in~$\mathcal{T}_D$.

\begin{theorem}\label{main_theorem}
For any $n\in\mathbb{N}$ and any non-increasing degree sequence $D\in\mathbb{N}_0^n$ such that $\mathcal{T}_D\neq\varnothing$,
the greedy tree~$GT_D$ attains the minimum Sombor index in~$\mathcal{T}_D$.
\end{theorem}

The remainder of the paper shall be structured as follows. 
In Section \ref{pseudo} we will define two auxiliary graph-related concepts: 
the vertex score and the pseudo-Sombor index, and prove some of their basic properties. 
Section \ref{main_section} will combine these two concepts with edge switching 
in order to complete the proof of Theorem~\ref{main_theorem}.

\section{Vertex score and pseudo-Sombor index}\label{pseudo}

In this and the following section, 
we assume that $D = (d_1, d_2, \ldots, d_n)$ is 
a fixed non-increasing sequence of non-negative integers for some $n \in \mathbb{N}, \, n \ge 2$,
such that $\mathcal{T}_D\neq\varnothing$.
Recall that we assume $\deg(u)=d_u$ for each vertex $u\in\{1,\dots,n\}$ of each tree in~$\mathcal{T}_D$.
Given the fact that the set $\mathcal{T}_D$ is finite, it is clear that the set
\[
    Z_D = \{ \mathrm{SO}(T) \colon T \in \mathcal{T}_D \}
\]
must also be non-empty and finite. Let $z_D^{(1)} = \min Z_D$ be its smallest element. Now, we have two possibilities: either the set $Z_D$ contains only the element $z_D^{(1)}$, or it has at least two elements, in which case we will denote its second smallest element via $z_D^{(2)} = \min \left( Z_D \setminus \left\{ z_D^{(1)} \right\} \right)$. By taking this into consideration, we are able to define a sufficiently small constant $q_D > 0$ via the expression
\[
    q_D = \begin{cases}
        \dfrac{1}{2n},& |Z_D| = 1,\\
        \min \left\{ \dfrac{1}{2n}, \dfrac{z_D^{(2)} - z_D^{(1)}}{4n^3 \sqrt{2}} \right\},& |Z_D| \ge 2.
    \end{cases}
\]

Furthermore, we will rely on $q_D$ in order to define the vertex score $\mathrm{scr}(u)$ for each vertex $u$ in the following manner:
\begin{equation}\label{score_def}
    \mathrm{scr}(u) = \deg(u) - u \, q_D.
\end{equation}
We can imagine the vertex score as a property very similar to the degree, albeit slightly smaller. 
Unlike the degrees, the vertex scores satisfy the strict monotonicity property that we shall heavily rely on afterwards. 
This conclusion is disclosed in the following lemma.

\begin{lemma}\label{monotonicity}
For each tree in $\mathcal{T}_D$ we have
    \[
        \mathrm{scr}(1) > \cdots > \mathrm{scr}(n) > 0 .
    \]
\end{lemma}
\begin{proof}
    Let $u$ and $v$ be two distinct vertices of the given tree such that $u < v$. Directly from the definition of the set $\mathcal{T}_D$, we obtain that $\deg(u) \ge \deg(v)$. By virtue of Eq.\ (\ref{score_def}), we see that
    \[
        \mathrm{scr}(u) - \mathrm{scr}(v) = (\deg(u) - \deg(v)) + (v-u) \, q_D .
    \]
    Given the fact that $v - u > 0$ and $q_D > 0$, it immediately follows that
    \[
        \mathrm{scr}(u) > \mathrm{scr}(v) .
    \]
    
    In order to finalize the proof, it is sufficient to show that $\mathrm{scr}(n) > 0$. However, since $n \ge 2$, it is clear that $\deg(n) \ge 1$, as well as $q_D \le \dfrac{1}{2n}$, which further implies
    \[
        \mathrm{scr}(n) = \deg(n) - n \, q_D \ge 1 - n \cdot \frac{1}{2n} = \frac{1}{2} > 0 ,
    \]
    as desired.
\end{proof}

In fact, the whole point of using vertex scores instead of their degrees is to avoid having the same value correspond to two different vertices. By relying on a vertex score instead of its degree, we define the auxiliary pseudo-Sombor index $\mathrm{pSO}(T)$ of an arbitrary tree $T \in \mathcal{T}_D$ as follows:
\[
    \mathrm{pSO}(T) = \sum_{u \sim v} \sqrt{\mathrm{scr}(u)^2 + \mathrm{scr}(v)^2} .
\]
Due to the fact that the positive constant $q_D$ is chosen to be fairly small, it makes sense that the difference between the pseudo-Sombor index and the Sombor index is also relatively small. 
We demonstrate this fact in the following lemma.
\begin{lemma}\label{approx_lemma}
If $| Z_D | \ge 2$, then for any tree $T \in \mathcal{T}_D$ we have
    \[
        \mathrm{SO}(T) - \frac{z_D^{(2)} - z_D^{(1)}}{2} < \mathrm{pSO}(T) < \mathrm{SO}(T) .
    \]
\end{lemma}
\begin{proof}
    Given the fact that all the vertex scores are positive and smaller than the corresponding degrees, the inequality $\mathrm{pSO}(T) < \mathrm{SO}(T)$ is obvious. Hence, we only need to prove that
    \[
        \mathrm{SO}(T) - \mathrm{pSO}(T) < \frac{z_D^{(2)} - z_D^{(1)}}{2} .
    \]
    
    To start, it is easy to see that for any vertex $u = \overline{1, n}$, we necessarily have
    \[
        \deg(u) - \mathrm{scr}(u) = u \, q_D \le n \, \dfrac{z_D^{(2)} - z_D^{(1)}}{4n^3 \sqrt{2}} \le \dfrac{z_D^{(2)} - z_D^{(1)}}{4n^2 \sqrt{2}} ,
    \]
    as well as
    \[
        \deg(u) + \mathrm{scr}(u) \le 2 \deg (u) < 2n,
    \]
    which immediately gives
    \begin{equation}\label{aux_1}
        \deg(u)^2 - \mathrm{scr}(u)^2 = (\deg(u) - \mathrm{scr}(u))(\deg(u) + \mathrm{scr}(u)) < \frac{z_D^{(2)} - z_D^{(1)}}{2n \sqrt{2}} .
    \end{equation}
    Furthermore, for any two vertices $u$ and $v$, we quickly obtain
    \begin{equation}\label{aux_2}
        \sqrt{\deg(u)^2 + \deg(v)^2} + \sqrt{\mathrm{scr}(u)^2 + \mathrm{scr}(v)^2} > \sqrt{\deg(u)^2 + \deg(v)^2} \ge \sqrt{2} ,
    \end{equation}
    given the fact that no vertex degree can be below one.
    
    Now, for any two adjacent vertices $u \sim v$, we conclude that
    \begin{align*}
        \sqrt{\deg(u)^2 + \deg(v)^2} &- \sqrt{\mathrm{scr}(u)^2 + \mathrm{scr}(v)^2} =\\
        &= \frac{(\deg(u)^2 + \deg(v)^2) - (\mathrm{scr}(u)^2 + \mathrm{scr}(v)^2)}{\sqrt{\deg(u)^2 + \deg(v)^2} + \sqrt{\mathrm{scr}(u)^2 + \mathrm{scr}(v)^2}}\\
        &= \frac{(\deg(u)^2 - \mathrm{scr}(u)^2) + (\deg(v)^2 - \mathrm{scr}(v)^2)}{\sqrt{\deg(u)^2 + \deg(v)^2} + \sqrt{\mathrm{scr}(u)^2 + \mathrm{scr}(v)^2}} .
    \end{align*}
    By implementing both Eq.\ (\ref{aux_1}) and Eq.\ (\ref{aux_2}), we promptly reach
    \begin{align*}
        \sqrt{\deg(u)^2 + \deg(v)^2} - \sqrt{\mathrm{scr}(u)^2 + \mathrm{scr}(v)^2} < \frac{\dfrac{z_D^{(2)} - z_D^{(1)}}{2n \sqrt{2}} + \dfrac{z_D^{(2)} - z_D^{(1)}}{2n \sqrt{2}}}{\sqrt{2}} = \dfrac{z_D^{(2)} - z_D^{(1)}}{2n} .
    \end{align*}
    Finally, we get
    \begin{align*}
        \mathrm{SO}(T) - \mathrm{pSO}(T) &= \sum_{u \sim v} \left( \sqrt{\deg(u)^2 + \deg(v)^2} - \sqrt{\mathrm{scr}(u)^2 + \mathrm{scr}(v)^2} \right)\\
        &< \sum_{u \sim v} \frac{z_D^{(2)} - z_D^{(1)}}{2n} = \frac{z_D^{(2)} - z_D^{(1)}}{2n} \, (n-1) < \frac{z_D^{(2)} - z_D^{(1)}}{2} ,
    \end{align*}
    thus completing the proof.
\end{proof}

The approximation obtained in Lemma \ref{approx_lemma} can now be used to show a key property of the pseudo-Sombor index that plays a central role in the proof of Theorem~\ref{main_theorem}. This property in given in the next lemma.
\begin{lemma}\label{cool_lemma}
If a tree~$T$ has the minimum pseudo-Sombor index in~$\mathcal{T}_D$, 
then $T$ also has the minimum Sombor index in~$\mathcal{T}_D$.
\end{lemma}
\begin{proof}
    First of all, in case we have $|Z_D| = 1$, it can immediately be seen that all the trees in $\mathcal{T}_D$ must have the same Sombor index, hence any tree attains the minimum Sombor index value. In the remainder of the proof, we will suppose that $|Z_D| \ge 2$.

    Let $T$ be a tree that attains the minimum pseudo-Sombor index in~$\mathcal{T}_D$, and 
let $T'$ be a tree that attains the minimum Sombor index in~$\mathcal{T}_D$. 
From Lemma~\ref{approx_lemma}, we know that $\mathrm{pSO}(T') < \mathrm{SO}(T') = z_D^{(1)}$, 
which further implies that $\mathrm{pSO}(T)\leq\mathrm{pSO}(T')<z_D^{(1)}$.
From Lemma~\ref{approx_lemma} we also have 
$\mathrm{SO}(T) - \dfrac{z_D^{(2)} - z_D^{(1)}}{2} < \mathrm{pSO}(T)$, 
which means that 
$$
\mathrm{SO}(T) < z_D^{(1)} + \dfrac{z_D^{(2)} - z_D^{(1)}}{2} = \dfrac{z_D^{(2)} + z_D^{(1)}}{2} < z_D^{(2)}.
$$ 
Since $z_D^{(1)}$ is the only possible value of Sombor index from~$Z_D$ that is smaller than $z_D^{(2)}$, 
we obtain that $\mathrm{SO}(T) = z_D^{(1)}$, 
meaning that $T$ has the minimum Sombor index in~$\mathcal{T}_D$.
\end{proof}

\section{Greedy trees}\label{main_section}

As a direct consequence of Lemma \ref{cool_lemma}, 
we see that in order to demonstrate that $GT_D$ attains the minimum value of Sombor index, 
it is sufficient to prove that $GT_D$ attains the minimum value of the pseudo-Sombor index. 
In this section we do that by showing that 
for any tree $T\in\mathcal{T}_D$ with $T\not\cong GT_D$
there is another tree $T'\in\mathcal{T}_D$ such that $pSO(T')<pSO(T)$.

For convenience, we will assume that all of the trees from~$\mathcal{T}_D$ are rooted, 
with the root fixed at the vertex~$1$. 
In this case, the root must have the highest score and no two vertices can have the same score, 
as shown in Lemma~\ref{monotonicity}. 
The {\em level} of a vertex will denote its distance to the root~$1$.
Bearing this in mind, we now disclose three helpful lemmas that together show that, apart from~$GT_D$, 
no other tree can attain the minimum value of the pseudo-Sombor index in~$\mathcal{T}_D$.

\begin{lemma}\label{switching}
Let $T \in \mathcal{T}_D$ be a tree containing four distinct vertices $u, v, w, t$ such that
    \begin{align*}
        u \sim v, && w \sim t, && u \nsim w, && v \nsim t .
    \end{align*}
    Suppose that the graph obtained from $T$ by deleting the edges $\{u, v\}$ and $\{w, t\}$ and adding the edges $\{u, w\}$ and $\{ v, t \}$ is a tree, and denote it by $T_1$. In that case we have $T_1 \in \mathcal{T}_D$, as well as
    \[
        \mathrm{pSO}(T) > \mathrm{pSO}(T_1) \quad \iff \quad (\mathrm{scr}(u) - \mathrm{scr}(t))(\mathrm{scr}(w) - \mathrm{scr}(v)) > 0 .
    \]
\end{lemma}
\begin{proof}
    First of all, given the fact that $T_1$ is guaranteed to be a tree and $T$ and $T_1$ obviously have the same degree sequence, it is clear that $T_1 \in \mathcal{T}_D$. Furthermore, pseudo-Sombor indices of these two trees will have the same summands, except for the terms that correspond to the deleted and newly added edges. With this in mind, we quickly get
    \begin{align*}
        \mathrm{pSO}(T) - \mathrm{pSO}(T_1) &= \sqrt{\mathrm{scr}(u)^2 + \mathrm{scr}(v)^2} + \sqrt{\mathrm{scr}(w)^2 + \mathrm{scr}(t)^2}\\
        &- \sqrt{\mathrm{scr}(u)^2 + \mathrm{scr}(w)^2} - \sqrt{\mathrm{scr}(v)^2 + \mathrm{scr}(t)^2} .
    \end{align*}
    It follows that $\mathrm{pSO}(T) > \mathrm{pSO}(T_1)$ is equivalent to
    \[
        \sqrt{\mathrm{scr}(u)^2 + \mathrm{scr}(v)^2} + \sqrt{\mathrm{scr}(w)^2 + \mathrm{scr}(t)^2} > \sqrt{\mathrm{scr}(u)^2 + \mathrm{scr}(w)^2} + \sqrt{\mathrm{scr}(v)^2 + \mathrm{scr}(t)^2} ,
    \]
    which, after squaring, becomes equivalent to
    \[
        (\mathrm{scr}(u)^2 + \mathrm{scr}(v)^2)(\mathrm{scr}(w)^2 + \mathrm{scr}(t)^2) > (\mathrm{scr}(u)^2 + \mathrm{scr}(w)^2)(\mathrm{scr}(v)^2 + \mathrm{scr}(t)^2) .
    \]
    Expanding the above expressions, we conclude that
    \begin{alignat*}{2}
        && \mathrm{pSO}(T) &> \mathrm{pSO}(T_1)\\
        \iff \quad && \mathrm{scr}(u)^2 \, \mathrm{scr}(w)^2 + \mathrm{scr}(v)^2 \, \mathrm{scr}(t)^2 &> \mathrm{scr}(u)^2 \, \mathrm{scr}(v)^2 + \mathrm{scr}(w)^2 \, \mathrm{scr}(t)^2\\
        \iff \quad && (\mathrm{scr}(u)^2 - \mathrm{scr}(t)^2)(\mathrm{scr}(w)^2 - \mathrm{scr}(v)^2) &> 0 .
    \end{alignat*}
    Since all vertex scores are positive, it is trivial to see that
    \[
        (\mathrm{scr}(u)^2 - \mathrm{scr}(t)^2)(\mathrm{scr}(w)^2 - \mathrm{scr}(v)^2) > 0
    \]
    is further equivalent to
    \[
        (\mathrm{scr}(u) - \mathrm{scr}(t))(\mathrm{scr}(w) - \mathrm{scr}(v)) > 0.
\vspace{-33.3pt}        
    \]
\end{proof}

We can now use the switching mechanism from Lemma~\ref{switching} 
to construct a tree with a smaller pseudo-Sombor index whenever we are given a tree different from~$GT_D$. 
The necessary constructions are given in the following two lemmas.

\begin{lemma}\label{main_lemma_1}
If a tree $T \in \mathcal{T}_D$ contains vertices $\alpha$ and~$\beta$ 
such that $\alpha$ is at a greater level than $\beta$, but $\mathrm{scr}(\alpha) > \mathrm{scr}(\beta)$,
then $T$ cannot attain the minimum pseudo-Sombor index in~$\mathcal{T}_D$.
\end{lemma}

\begin{proof}
Let $j$ be the minimum index such that
each vertex on level~$k$, for each $0\leq k\leq j-1$, has a higher score 
than any vertex belonging to a level greater than~$k$,
but such that there exists a vertex~$\beta$ on level~$j$ and a vertex~$\alpha$ on a level greater than~$j$
with $\mathrm{scr}(\beta) < \mathrm{scr}(\alpha)$.
Since the root~$1$ has the highest score, we have that $j\geq 1$,
so that $\beta$ has a parent, which we denote by~$\gamma$.
In order to make the proof more concise, we will divide it into two cases depending on whether $\beta$ is the parent of~$\alpha$.

    \bigskip\noindent
    \emph{Case $\beta$ is the parent of $\alpha$}.\quad
    In this case, we clearly have that $\deg(\beta) \ge 2$. 
    Since $\mathrm{scr}(\alpha) > \mathrm{scr}(\beta)$, 
    Lemma \ref{monotonicity} tells us that $\alpha < \beta$, 
    hence $\deg(\alpha) \ge 2$ as well. 
    (Recall that the degrees are ordered in a non-increasing order in~$D$.)
    This means that the vertex $\alpha$ must have at least one child, which we shall denote via $\delta$. We now get that
    \begin{align*}
        \alpha \sim \delta, && \beta \sim \gamma, && \gamma \nsim \alpha, && \beta \nsim \delta .
    \end{align*}
    If we construct a graph $T_1$ from $T$ by deleting the edges $\{ \alpha, \delta \}$ and $\{ \beta, \gamma \}$ and adding the new edges $\{\gamma, \alpha\}$ and $\{ \beta, \delta \}$, we see that this graph must be a tree from the set $\mathcal{T}_D$. According to Lemma \ref{switching}, we obtain
    \[
        \mathrm{pSO}(T) > \mathrm{pSO}(T_1) \quad \iff \quad (\mathrm{scr}(\alpha) - \mathrm{scr}(\beta))(\mathrm{scr}(\gamma) - \mathrm{scr}(\delta)) > 0 .
    \]
    We have $\mathrm{scr}(\alpha) > \mathrm{scr}(\beta)$ by assumption, 
    while $\mathrm{scr}(\gamma) > \mathrm{scr}(\delta)$ also holds 
    since $\gamma$ is from level~$j-1$ and $\delta$ is from a greater level than~$\gamma$.
    Thus, $\mathrm{pSO}(T_1) < \mathrm{pSO}(T)$, so $T$ cannot attain the minimum pseudo-Sombor index in~$\mathcal{T}_D$.

    \bigskip\noindent
    \emph{Case $\beta$ is not the parent of $\alpha$}.\quad
    In this case, let $\delta$ be the parent of $\alpha$ that is located on some level greater than $j-1$. Here, it is clear that the vertices $\alpha, \beta, \gamma, \delta$ are all mutually distinct. Moreover, we have that $\alpha \sim \delta$ and $\beta \sim \gamma$, but $\gamma \nsim \alpha$. However, $\delta$ and $\beta$ may or may not be adjacent. It is easy to see that these two vertices are adjacent if and only if $\beta$ is the parent of $\delta$. These two scenarios shall yield two different construction patterns for $T_1$. For this reason, we shall divide the given case into two further subcases.

    \medskip\noindent
    \emph{Subase $\beta$ is not the parent of $\delta$}.\quad
    In this subcase, we get
    \begin{align*}
        \alpha \sim \delta, && \beta \sim \gamma, && \gamma \nsim \alpha, && \beta \nsim \delta .
    \end{align*}
    As in the previous case, if we construct a graph $T_1$ from $T$ by deleting the edges $\{ \alpha, \delta \}$ and $\{ \beta, \gamma \}$ and adding the new edges $\{\gamma, \alpha\}$ and $\{ \beta, \delta \}$, it can be easily seen that this graph must be a tree from the set $\mathcal{T}_D$. By virtue of Lemma \ref{switching}, we have
    \[
        \mathrm{pSO}(T) > \mathrm{pSO}(T_1) \quad \iff \quad (\mathrm{scr}(\alpha) - \mathrm{scr}(\beta))(\mathrm{scr}(\gamma) - \mathrm{scr}(\delta)) > 0 .
    \]
    As in the previous case, we have that $\mathrm{scr}(\alpha) > \mathrm{scr}(\beta)$, and $\mathrm{scr}(\gamma) > \mathrm{scr}(\delta)$ must also be true since $\gamma$ is from level $j-1$ and $\delta$ is from a greater level. 
    Thus, $\mathrm{pSO}(T_1) < \mathrm{pSO}(T)$.

    \medskip\noindent
    \emph{Subcase $\beta$ is the parent of $\delta$}.\quad
    In this subcase, 
    we have that $\gamma$ is the parent of $\beta$, which is the parent of $\delta$, which is the parent of $\alpha$. 
    Since $\deg(\beta)\geq 2$ and $\alpha<\beta$, it follows that $\deg(\alpha) \ge 2$, as already observed. 
    Thus, $\alpha$ must have at least one child, and we will name one of them as~$\varepsilon$. 
    Now, we have
    \begin{align*}
        \alpha \sim \varepsilon, && \beta \sim \gamma, && \gamma \nsim \alpha, && \beta \nsim \varepsilon .
    \end{align*}
    If we construct a graph $T_1$ from $T$ by deleting the edges $\{ \alpha, \varepsilon \}$ and $\{ \beta, \gamma \}$ and adding the new edges $\{\gamma, \alpha \}$ and $\{ \beta, \varepsilon \}$, it can be quickly noticed that this graph must be a tree from the set $\mathcal{T}_D$. Furthermore, Lemma \ref{switching} gives us
    \[
        \mathrm{pSO}(T) > \mathrm{pSO}(T_1) \quad \iff \quad (\mathrm{scr}(\alpha) - \mathrm{scr}(\beta))(\mathrm{scr}(\gamma) - \mathrm{scr}(\varepsilon)) > 0 .
    \]
    Now, it is clear that $\mathrm{scr}(\gamma) > \mathrm{scr}(\varepsilon)$, 
    since $\gamma$ is from level $j-1$ and $\varepsilon$ is from a greater level, 
    which again implies that $\mathrm{pSO}(T_1) < \mathrm{pSO}(T)$.
\end{proof}

As a direct consequence of Lemma \ref{main_lemma_1}, 
we see that a tree $T \in \mathcal{T}_D$ with the minimum value of the pseudo-Sombor index satisfies the property that 
whenever a vertex $\alpha$ is at a greater level than a vertex $\beta$, then $\mathrm{scr}(\alpha)<\mathrm{scr}(\beta)$.
We will now show that the scores of vertices at the same level are aligned according to the scores of their parents.

\begin{lemma}\label{main_lemma_2}
If a tree $T \in \mathcal{T}_D$ contains 
two vertices $\alpha$ and $\beta$ on the same level with $\mathrm{scr}(\alpha) > \mathrm{scr}(\beta)$, 
such that $\alpha$ has a child~$\gamma$ and $\beta$ has a child~$\delta$ with $\mathrm{scr}(\gamma) < \mathrm{scr}(\delta)$, 
then $T$ cannot attain the minimum pseudo-Sombor index in~$\mathcal{T}_D$.
\end{lemma}

\begin{proof}
    It is clear that
    \begin{align*}
        \alpha \sim \gamma, && \beta \sim \delta, && \alpha \nsim \delta, && \beta \nsim \gamma .
    \end{align*}
    If $T_1$ is obtained from $T$ 
    by deleting the edges $\{ \alpha, \gamma \}$ and $\{ \beta, \delta \}$ and 
    adding the new edges $\{\alpha, \delta \}$ and $\{ \beta, \gamma \}$, 
    then it is easy to see that $T_1$ is a tree from $\mathcal{T}_D$ as well. 
    From Lemma \ref{switching}
    \[
        \mathrm{pSO}(T) > \mathrm{pSO}(T_1) \quad \iff \quad (\mathrm{scr}(\alpha) - \mathrm{scr}(\beta))(\mathrm{scr}(\delta) - \mathrm{scr}(\gamma)) > 0,
    \]
    which by the above assumptions implies that $\mathrm{pSO}(T_1) < \mathrm{pSO}(T)$.
    Hence $T$ does not attain the minimum pseudo-Sombor index in~$\mathcal{T}_D$.
\end{proof}

Taking into consideration both Lemma \ref{main_lemma_1} and Lemma \ref{main_lemma_2}, 
we see that the only way for a tree $T \in \mathcal{T}_D$ to attain the minimum pseudo-Sombor index is 
if the children of the root have the highest possible scores, 
then the children of the highest-scored child have the highest possible scores, etc. 
In other words, if we select the children in such a way that the ones with the higher scores go first, 
the tree~$T$ must be such that its breadth-first traversal yields a strictly decreasing sequence of scores. 
Recalling that $\mathrm{scr}(1)>\dots>\mathrm{scr}(n)$ by Lemma~\ref{monotonicity}, 
we see that such tree is actually the greedy tree~$GT_D$,
and this is the only tree that attains the minimum pseudo-Sombor index in~$\mathcal{T}_D$. 
We can now complete the proof of Theorem \ref{main_theorem}.

\bigskip\noindent
\emph{Proof of Theorem \ref{main_theorem}}.\quad
If $n = 1$, then $D=(0)$ is
the only non-increasing sequence of non-negative integers in $\mathbb{N}_0^1$ for which $\mathcal{T}_D \neq \varnothing$.
In such case we actually have $\mathcal{T}_D=\{K_1\}$.
Since $K_1$ is also a greedy tree,
the minimum Sombor index is clearly attained by the greedy tree in this case.

For $n\geq 2$, suppose that $D \in \mathbb{N}_0^n$ is 
an arbitrarily chosen non-increasing sequence of non-negative integers such that $\mathcal{T}_D\neq\varnothing$. 
Now, Lemmas \ref{main_lemma_1} and~\ref{main_lemma_2} guarantee that 
the greedy tree $GT_D$ attains the minimum pseudo-Sombor index in~$\mathcal{T}_D$.
Lemma~\ref{cool_lemma} then dictates that $GT_D$ must also attain the minimum Sombor index in~$\mathcal{T}_D$, 
which completes the proof.
\hfill\qed

\end{document}